\documentclass[12pt,reqno,a4paper,dvipsnames]{amsart}
\usepackage[utf8]{inputenc}
\usepackage[dvipsnames]{xcolor}

\title[Polynomial Wiener--Wintner]{Quantitative Polynomial Wiener--Wintner Theorems}

\author[Becker]{Lars Becker}
\address{Department of Mathematics, Princeton University, Fine Hall, Washington Road,
Princeton NJ, 08544-1000, USA}
\email{lbecker@math.princeton.edu}

\author[Jamneshan]{Asgar Jamneshan}
\address{Mathematical Institute, 
	University of Bonn,
	Endenicher Allee 60, 53115, Bonn,
	Germany}
\email{ajamnesh@math.uni-bonn.de}

\author[Thiele]{Christoph Thiele}
 \address{Mathematical Institute, 
	University of Bonn,
	Endenicher Allee 60, 53115, Bonn,
	Germany}
\email{thiele@math.uni-bonn.de}

\date{\today}

\usepackage{mathtools}
\usepackage{amssymb}
\usepackage{amsthm}
\usepackage{amsmath}
\usepackage{graphicx}
\usepackage{enumitem}
\usepackage[hidelinks, colorlinks=false]{hyperref}

\theoremstyle{plain}
\newtheorem{theorem}{Theorem}
\newtheorem{lemma}[theorem]{Lemma}
\newtheorem{prop}[theorem]{Proposition}
\newtheorem{cor}[theorem]{Corollary}
\theoremstyle{definition}

\newtheorem{remark}[theorem]{Remark}

\numberwithin{theorem}{section}
\numberwithin{equation}{section}

\newcommand{\R}{\mathbb{R}}

\newcommand{\N}{\mathbb{N}}
\newcommand{\AD}{D}

\DeclareMathOperator{\supp}{\operatorname{supp}}

\begin{document}

\begin{abstract}
We prove quantitative polynomial Wiener--Wintner theorems in a very general setup, including measure-preserving actions of nilpotent Lie groups. Our results apply both to ergodic averages and to averages with singular integral weights. The proof relies on the generalized polynomial Carleson theorem developed in the companion paper by van Doorn, Srivastava, and the authors. 
\end{abstract}

\maketitle

\section{Introduction}
\subsection{Wiener--Wintner theorems}
Wiener and Wintner \cite{wiener-wintner} established the following refinement of Birkhoff's pointwise ergodic theorem. 

\begin{theorem}
\label{t:wiener-wintner}
    Let $(X,\nu,T)$ be a measure-preserving system, and let $f \in L^1(X,\nu)$. Then there exists a full measure set $X_f \subset X$ such that for every $\theta \in [0,2\pi]$ and $x \in X_f$, the averages
    \[
        \frac{1}{2N+1}\sum_{n=-N}^N e(\theta n)\, f(T^n x),
    \]
    converge as $N \to \infty$, where $e(t) \coloneqq e^{it}$. 
\end{theorem}

The essential feature of this result is that the exceptional set $X_f$ can be chosen \emph{uniformly} for all frequencies $\theta$, whereas Birkhoff's theorem only guarantees the existence of such sets for individual $\theta$. 

Define the $r$-variation of a function $a$ on the interval $(0,\infty)$ by
\[
    \|a\|_{V^r}
    := \sup_{N,\, 0 <t_0 < t_1 < \dots < t_N}
    \Big( \sum_{n=1}^N |a(t_n) - a(t_{n-1})|^r \Big)^{1/r}.
\]
In this paper we obtain the following variant of Theorem \ref{t:wiener-wintner}.

\begin{theorem}\label{t:main}
    Let $p\in(1,\infty]$, $r > 2$, and $d \ge 1$. Let $G$ be a homogeneous Lie group with Haar measure $\mu$, let $(X,\nu)$ be a probability space, and let $T \colon G \times X \to X$
    be a measure-preserving action of $G$ on $(X,\nu)$. For each $f \in L^p(X, \nu)$ there exists a set $X_f \subset X$ of full measure such that for all $x \in X_f$, 
    \begin{equation*}
        \sup_{Q} \Big\|\frac{1}{\mu(B(0,R))} \int_{B(0,R)} 
        f(T^g x)\, e(Q(g))\, \mathrm{d}\mu(g)\Big\|_{V^r_R} < \infty,
    \end{equation*}
     where the supremum is taken over all measurable Leibman polynomials $Q$ of degree at most $d$. 
\end{theorem}

We refer to Section \ref{ss:lie} for the definitions. Our methods also apply to averages with singular integral kernels as weights, see Theorem \ref{t:sing} below. 

Theorem \ref{t:main} generalizes the classical Wiener--Wintner theorem in three directions. Firstly, it allows for measure-preserving actions of homogeneous Lie groups (such groups are always nilpotent). Secondly, it allows all polynomial, rather than just linear, phases. Thirdly, it strengthens Theorem \ref{t:wiener-wintner} by giving finiteness of the $r$-variation in the parameter $R$, thus quantifying the rate of convergence uniformly in the polynomial phase. 

Theorem \ref{t:main} follows by the Calderón transference principle from a polynomial Carleson theorem on homogeneous Lie groups with an additional variation norm in radial truncations of the kernel, see Corollary \ref{c:hom Lie} below. The group properties play only a secondary role in the proof.  What is relevant is the structure of the group as a doubling metric measure space. For this reason, we now state a more general main theorem in that setup.

\subsection{Main result on metric measure spaces}\label{ss:setup results} 

In what follows we work in a slightly less general setup than in the companion publication \cite{beckeretal}, whose main result we will apply to prove our main theorem. We detail the differences in Section \ref{s:SmoTrun}.

\subsubsection{Metric measure space}
A $\AD$-dimensional metric measure space $(X,\rho,\mu)$ is a complete metric space $(X,\rho)$ equipped with a Borel measure $\mu$ satisfying for all  $x \in X$ and $R > 0$, the condition
\begin{equation}\label{e:ADregular}
    \mu(B(x,R)) = C R^{\AD}.
\end{equation}

\subsubsection{Modulation functions}
A collection $\mathcal{Q}$ of real-valued continuous functions on $(X,\rho,\mu)$ is called compatible if the following conditions are satisfied. For any ball $B$ in $X$ and $f, g \in \mathcal{Q}$, denote 
\begin{equation}\label{e:osccontrol}
    d_B(f,g) := \sup_{x,y\in B} |f(x)-f(y)-g(x)+g(y)|. 
\end{equation}
\begin{itemize}
    \item[(1)] There exists $x_0 \in X$ with $Q(x_0) = 0$ for all $Q \in \mathcal{Q}$.
    \item[(2)] For any ball $B$ the function $d_B$ is a metric on $\mathcal{Q}$. 
    \item[(3)] For any balls $B_1 = B(x_1,R)$ and $B_2 = B(x_2,2R)$ with $x_1 \in B_2$ 
    \begin{equation}\label{e:firstdb}
        d_{B_2} \le C d_{B_1}\,.
    \end{equation}
    \item[(4)] For any balls $B_1 = B(x_1,R)$, $B_2 = B(x_2, C R)$ with $B_1 \subset B_2$ 
    \begin{equation}\label{e:seconddb}
        2\, d_{B_1} \le d_{B_2}\,.
    \end{equation}
    \item[(5)] For any ball $B$ and every $d_B$-ball $\widetilde{B}$ of radius $2R$ in $\mathcal{Q}$, there exists a collection of at most $C$ many $d_B$-balls of radius $R$ covering $\widetilde{B}$.
\end{itemize}

\subsubsection{Cancellation} A compatible collection $\mathcal{Q}$ is called $\varepsilon$-cancellative if for every ball $B$ in $X$ of radius $R$, every Lipschitz function $\varphi \colon X \to \mathbb{C}$ supported in $B$, and all $f,g \in \mathcal{Q}$,
\begin{equation}\label{e:vdc cond}
    \left|\int_B e((f-g)(x))\varphi(x)\, \mathrm{d}\mu\right|
    \le C R^{\AD}\, \|\varphi\|_{C^{0,1}(B)}\, (1 + d_B(f,g))^{-\varepsilon},
\end{equation}
where $\|\cdot\|_{C^{0,1}(B)}$ is the inhomogeneous Lipschitz norm normalized as follows 
\[
    \|\varphi\|_{C^{0,1}(B)} =  \sup_{x \in B} |\varphi(x)| +   R \sup_{\substack{x,y \in B \\ x \neq y}} \frac{|\varphi(x) - \varphi(y)|}{\rho(x,y)}\,.
\]
\subsubsection{Singular integral kernels}
For $0<\alpha\leq 1$, an $\alpha$-kernel $K$ on $X$ is a measurable function
\begin{equation*}
    K \colon X \times X \to \mathbb{C}
\end{equation*}
such that for all $x, y', y \in X$ with $x \neq y$
\begin{equation}\label{e:kernel size}
    |K(x,y)| \le \rho(x,y)^{-\AD}\,,
\end{equation}
and if $2\rho(y,y') \le \rho(x,y)$, then
\begin{equation}\label{e:kernel y smooth}
    |K(x,y) - K(x,y')| + |K(y,x) - K(y',x)|
    \le
    \left(\frac{\rho(y,y')}{\rho(x,y)}\right)^{\alpha}
    \rho(x,y)^{-\AD}\,.
\end{equation}
The kernel $K$ satisfies the cancellation condition if for all $x \in X$ and $0 < R_1 < R_2$,
\begin{equation}\label{e:CZcancellation}
    \int_{B(x,R_2) \setminus B(x,R_1)} K(x,y) \, \mathrm{d}\mu(y)=
    \int_{B(x,R_2) \setminus B(x,R_1)} K(y,x) \, \mathrm{d}\mu(y) = 0.
\end{equation}

\subsubsection{Statement of the main result}

We fix a $\AD$-dimensional metric measure space $(X, \rho, \mu)$ and a compatible $\varepsilon$-cancellative collection $\mathcal{Q}$ on $X$. The constants in the below statements may depend on this data.

We consider truncated, modulated singular integrals 
\begin{equation}\label{e:trsdef}
    S_u(K, Q, f)(x)
    := \int_{\rho(x,y) > u} K(x,y)\, f(y)\, e(Q(y))\, \mathrm{d}\mu(y),
\end{equation}
and truncated, modulated averages 
\begin{equation}\label{e:tradef}
    A_R(Q, f)(x)
    := \frac{1}{\mu(B(x,R))} \int_{B(x,R)} f(y)\, e(Q(y))\, \mathrm{d}\mu(y).
\end{equation}
Our most general result is a variational estimate, uniform in the modulation, for both averages and singular integrals. 

\begin{theorem}\label{t:metric}
Let $p \in (1,\infty)$, $r>2$, and $0<\alpha\leq 1$. Then there exists a constant $C > 0$ such that for all $f\in L^p(X,\mu)$, we have 
\begin{equation}\label{e:thm ShTru2}
    \left\| \sup_{Q \in \mathcal{Q}}
    \bigl\| A_R(Q, f) \bigr\|_{V^r_R} \right\|_p
    \le C \|f\|_p.
\end{equation}
Moreover, for every $\alpha$-kernel $K$ satisfying the cancellation condition \eqref{e:CZcancellation} it holds for all $f\in L^p(X,\mu)$ that 
\begin{equation}\label{e:thm ShTru1}
    \left\| \sup_{Q \in \mathcal{Q}}
    \bigl\| S_u(K, Q, f) \bigr\|_{V^r_u} \right\|_p
    \le C \|f\|_p.
\end{equation}
\end{theorem}

\subsection{Homogeneous Lie groups}
\label{ss:lie}

We now give the relevant definitions for homogeneous Lie groups\footnote{Background on homogeneous Lie groups can be found in e.g.~\cite{bonfiglioli,ledonne-book}.} which are our primary objects of interest. 

A homogeneous Lie group is a connected, simply connected Lie group whose Lie algebra is endowed with a family of automorphic dilations. Every homogeneous Lie group is isomorphic, as a Lie group, to a group of the form $(\mathbb{R}^n, \circ, \{\delta_\lambda\}_{\lambda>0})$, where the group law is given by
\[
x \circ y = x + y + P(x,y),
\]
with $P$ a polynomial map. The dilations form a one-parameter family of automorphisms of the form
\[
\delta_\lambda(x_1,\ldots,x_n)
    \coloneqq (\lambda^{\sigma_1} x_1, \ldots, \lambda^{\sigma_n} x_n),
\]
for some real numbers $1 \leq \sigma_1 \leq \cdots \leq \sigma_n$. The quantity
\[
\AD = \sum_{i=1}^n \sigma_i
\]
is called the homogeneous dimension of the group.

Every homogeneous Lie group is nilpotent and its Haar measure coincides with the Lebesgue measure $\mu$ on $\mathbb{R}^n$. By \cite[Theorem~2]{hebisch1990smooth}, there exists $\epsilon > 0$, depending on the group, such that the function
\begin{equation}\label{e:def homgp rho}
    \rho(x,y) = \inf\{\lambda > 0 : \delta_{\lambda^{-1}}(x \circ y^{-1}) \in U_\epsilon\},
\end{equation}
with
\[
U_\epsilon := \{x \in \mathbb{R}^n : \sum_{j=1}^n x_j^2 < \epsilon^2\},
\]
defines a right-invariant metric on the group. Moreover, the unit ball $B(0,1)$ with respect to $\rho$ coincides with the Euclidean ball $U_\epsilon$. 

On homogeneous Lie groups, we consider measurable Leibman polynomials \cite{leibman} as the modulation functions. For the convenience of the reader, we recall their definition. Let $G,H$ be groups and let $f\colon G\to H$ be a function. For $g\in G$, define the discrete derivative
\[
\Delta_g f(g')\coloneqq f(g'g^{-1}) \cdot f(g')^{-1}. 
\]
Then $f$ is said to be a polynomial map of degree $0$ if $\Delta_g f\equiv 1$ for all $g\in G$, and for $d\geq 1$, a polynomial map of degree $d$ if $\Delta_g f$ is polynomial of degree $d-1$ for all $g\in G$.

If $G,H$ are locally compact Polish groups, then it follows from \cite[Theorem 1.5]{auto-cont} that a measurable Leibman polynomial map between $G,H$ is automatically continuous. 
Moreover, real-valued, continuous Leibman polynomial maps on a homogeneous Lie group $(\mathbb{R}^n, \circ, \{\delta_\lambda\}_{\lambda>0})$ are precisely the homogeneous polynomials/polynomials in exponential coordinates (see \cite[Corollary 1.4, Remark 4.2]{ledonne-poly}, \cite[Remark 6.2]{poly-cohomology}). In particular, they are classical polynomial functions of potentially higher classical degree.

In Section \ref{s:lie}, we verify that all homogeneous Lie groups and their measurable Leibman polynomials satisfy the assumptions of Section \ref{ss:setup results}. Hence, we have the following consequence of Theorem \ref{t:metric}. 
 \begin{cor}\label{c:hom Lie}
Let $p \in (1,\infty)$, $r>2$, $0 < \alpha \le 1$, and $d\geq 1$. Let $G$ be a homogeneous Lie group with Haar measure $\mu$, and let $\mathcal{Q}$ be the collection of real-valued  measurable Leibman polynomial maps on $G$ of degree at most $d$. 
There exists a constant $C  > 0$ such that the following holds. For all $f \in L^p(G, \mu)$,
\begin{equation}\label{e:hom lie1}
    \left\| \sup_{Q \in \mathcal{Q}}
    \bigl\| A_R(Q, f) \bigr\|_{V^r_R} \right\|_p
    \le C \|f\|_p.
\end{equation}
Moreover, if $K$ is an $\alpha$-kernel on $G$ satisfying the cancellation condition~\eqref{e:CZcancellation}, then for all $f \in L^p(G,\mu)$,
\begin{equation}\label{e:hom lie2}
    \left\| \sup_{Q \in \mathcal{Q}}
    \bigl\| S_u(K, Q, f) \bigr\|_{V^r_u} \right\|_p
    \le C \|f\|_p.
\end{equation}
\end{cor}

Estimate \eqref{e:hom lie1} implies by the Calderón transference principle \cite{MR227354} the following quantitative version of Theorem \ref{t:main}, see Appendix \ref{transference}. 

\begin{theorem}\label{t:avg}
    Let $G$ be a homogeneous Lie group with Haar measure $\mu$, and let $p \in (1,\infty)$, $r > 2$, and $d \ge 1$. There exists a constant $C > 0$ such that the following holds. Let $(X,\nu)$ be a probability space and let $T \colon G \times X \to X$ be a measure-preserving action of $G$ on $(X,\nu)$. For each $f \in L^p(X, \nu)$,
    \begin{equation*}
        \Big\| \sup_{Q} \Big\|\frac{1}{\mu(B(0,R))} \int_{B(0,R)} 
        f(T^g x)\, e(Q(g))\, \mathrm{d}\mu(g)\Big\|_{V^r_R} \Big\|_{p} 
    \end{equation*}
    \begin{equation}\label{eq ww avg hom}
        \le C \|f\|_{p},
    \end{equation}
    where the supremum is taken over all measurable real-valued Leibman polynomials $Q$ of degree at most $d$. 
\end{theorem}

Note that Theorem \ref{t:avg} implies Theorem \ref{t:main}, as functions in $L^p$ are finite almost everywhere.

From Estimate \eqref{e:hom lie2} we obtain similarly the following result with a singular integral weight. We call a function $k$ an $\alpha$-convolution kernel on $G$ if the function $K(x,y) = k(x\circ y^{-1})$ is an $\alpha$-kernel satisfying the cancellation condition \eqref{e:CZcancellation}. 

\begin{theorem}
    \label{t:sing}
     In the situation of Theorem \ref{t:main}, for every $\alpha \in (0,1]$, there exists a constant $C > 0$ such that for every $\alpha$-convolution kernel on $G$,
    \begin{equation*}
        \Big\| \sup_{Q} \Big\|\int_{G\setminus B(0,R)} 
        f(T^g x)\, e(Q(g)) k(g) \, \mathrm{d}\mu(g)\Big\|_{V^r_R} \Big\|_{p} 
        \le C \|f\|_{p},
    \end{equation*}
    where the supremum is taken over all measurable Leibman polynomials $Q$ of degree at most $d$.
\end{theorem}

We record a further corollary, motivated from the formulation of the Wiener--Wintner theorem for amenable groups in \cite{ornstein-weiss,zk-return}. 

\begin{cor}\label{c:unitary}
    Let $n \in \mathbb{N}$ and $C > 0$.
    In the situation of Theorem \ref{t:main}, for each $f \in L^p(X, \nu)$, there exists a full measure set $X_f \subset X$ such that for all $x \in X_f$,
    \[
        \sup_{Q, \phi} \Big\| \frac{1}{\mu(B(0,R))}\int_{B(0,R)} f(T^g x) \phi(Q(g)) \, \mathrm{d}\mu(g) \Big\|_{V^r_R} < \infty,
    \]
     where the supremum is over all unitary groups $U$ of dimension at most $n$, all measurable quadratic Leibman polynomials $Q: G \to U$, and all smooth functions $\phi: U \to \mathbb{C}$ of $C^{n}$-norm at most $C$.
    The same holds in the singular integral case. 
\end{cor}

For nilpotent Lie groups, formulations of the Wiener--Wintner theorem using general unitary representations as in \cite{ornstein-weiss, zk-return} are equivalent to the formulation using just characters, as all finite-dimensional irreducible representations of connected nilpotent Lie groups are one dimensional. Corollary \ref{c:unitary} holds because similarly all quadratic polynomials $G \to U$ have image in a coset of a torus (and because one can lift torus-valued polynomials to real-valued ones, see Remark \ref{rem:lifting}). We prove this in Section \ref{s:quadratic} using results from \cite{quadratic}. It is natural to conjecture that the same is true for polynomials of higher degrees (cf.~Remark \ref{higher-degree}). This would imply Corollary \ref{c:unitary} for polynomials of general degree.

\subsection{Remarks} We comment on previous results, possible extensions of Theorem \ref{t:main}, and on the optimality of the assumptions.

\subsubsection{Variation estimates}
The method of using variation estimates to prove pointwise convergence results in ergodic theory via harmonic analysis originates in works of Bourgain \cite{Bourgain_arithmetic, bourgain}, and is by now a standard tool. It is well known that the range of variational exponents $r > 2$ is best possible, already for Birkhoff's theorem. This follows from comparison with Brownian motion, which has infinite $2$-variation almost surely, see for example the proof of Lemma 3.11 in \cite{Bourgain_arithmetic}.

In the context of the linear Wiener--Wintner theorem, this method was first pursued by Lacey and Terwilleger \cite{lacey-terwilleger}, who proved the singular integral version of Theorem \ref{t:main} for $G = \R$ and linear modulations, and by Oberlin et al. \cite[Appendix D]{oberlinetal}, who obtain a similar result as a corollary of the stronger variational Carleson theorem. A very recent application, including also some polynomial modulations and the discrete case, was given by Krause in \cite{krause2025}.

For the polynomial Wiener--Wintner theorem, the same reasoning leads to the variational truncation version of the polynomial Carleson theorem of Lie \cite{lie-quadratic, lie-polynomial}. Part of the motivation for the present paper was to demonstrate that such extensions follow, in a very general setup, from the generalized polynomial Carleson theorem in \cite{beckeretal}.
We emphasize that this theorem in \cite{beckeretal} is a formalized theorem, i.e., 
verified by computer \cite{becker2025blueprintformalizationcarlesonstheorem}.

\subsubsection{Range of exponents}

The use of Carleson's theorem, or some extension of it, restricts the range of exponents in Theorem \ref{t:metric} to $p > 1$, and also precludes a weak type estimate at the endpoint. If one is interested in a qualitative Wiener--Wintner theorem as in Theorem \ref{t:wiener-wintner}, that is mere convergence but no variation bound, then the result is also true for $p = 1$. This follows from the result for $p > 1$, using the maximal ergodic theorem and density of $L^p$, $p > 1$, in $L^1$. However, the variation norm bound cannot be extended to the case $p = 1$ by such arguments.

\subsubsection{Generality of the acting group}
There is a vast literature on qualitative generalizations of the Wiener--Wintner theorem, see, e.g., \cite{assani-survey,eisner-farkas} for an overview. We mention in particular \cite{ornstein-weiss, zk-return} who prove a (linear) Wiener--Wintner theorem for amenable groups, and \cite{lesigne2} who proves a polynomial Wiener--Wintner theorem on the integers. Our use of \cite{beckeretal} requires that $G$ be a doubling metric measure space, and this assumption is essential. Consequently, the largest class of Lie groups to which our result can apply is given by the nilpotent groups; see \cite[Section~4]{nevo} and the references therein. It may be worthwhile to investigate whether meaningful variants of Carleson's theorem persist in non-doubling settings. 

By contrast, the assumption that the space is $\AD$-dimensional is inessential and is imposed only to streamline certain arguments. It can be replaced by a small-boundary condition as in \cite{zk-var-trunc}.

Not every connected, simply connected nilpotent Lie group is homogeneous \cite{dyer}. Theorem \ref{t:main} therefore leaves open the (polynomial) Wiener--Wintner theorem for nilpotent Lie groups that are not homogeneous. In this setting, the existing literature suggests taking the variation norm in Theorem \ref{t:main} only over large balls, i.e.\ restricting to scales $t\in(1,\infty)$. We expect that our methods also yield this large-scale variant of Theorem \ref{t:main} on non-homogeneous nilpotent Lie groups.

Indeed, the assumptions of Section \ref{ss:setup results} are satisfied on every nilpotent Lie group for all sufficiently large balls. The only essential input in our proof is Lemma \ref{l: comp ball}, which holds for balls of radius at least $1$ in every nilpotent Lie group (and for any reasonable metric); see, for example, \cite{Karidi, MSW07}. One may then combine this with suitable large-scale variants of the results in \cite{beckeretal} and \cite{zk-var-trunc} to obtain the corresponding conclusions for large balls. These variants follow from a tedious but straightforward inspection of the arguments. We leave the details to the interested reader.

\subsubsection{Lattices}
All of our results are formulated for measure-preserving flows. It is natural to ask whether the methods extend to lattices in nilpotent Lie groups. Our approach does not directly extend except in the case of linear polynomials and the lattice $\mathbb{Z}^d$. In that case the Magyar--Stein--Wainger sampling principle \cite[Proposition 2.1]{Magyar+02} allows to transfer variation bounds with smooth cutoff functions from $\R^d$ to $\mathbb{Z}^d$. Using an approximation argument as in Section \ref{s:ShTruformSmoTru}, this implies $r$-variation bounds for averages over balls for sufficiently large $r$. On the other hand, already for non-linear polynomials and the lattice $\mathbb{Z}^d$, such transference arguments fail. An easier open problem is the quadratic Carleson theorem on the integers, we refer to \cite{KrauseRoos23, Krause24} for related work, and to \cite{krause2025} for an application towards discrete polynomial Wiener--Wintner theorems.

\subsection{Structure of the paper}
\label{s:overview}

Most of the paper is devoted to the proof of Theorem \ref{t:metric}. This is done in three steps, in Sections \ref{s:SmoTrun}, \ref{s:ShTruformSmoTru}, and \ref{s:sparse}. In the first step we apply the generalized polynomial Carleson theorem from \cite{beckeretal}, see Theorem \ref{t:gen Carleson}, to obtain a weaker variant of Theorem \ref{t:metric}, see Proposition \ref{p:SmoTrun}. It differs from Theorem \ref{t:metric} in two ways: It uses smooth cutoff functions to truncate the singular integrals or averages, and it only claims bounds in the smaller range of exponents $p \in (1,2)$. The application of Theorem \ref{t:gen Carleson} requires as input the estimate \eqref{e:nontanbound} for a  nontangential maximal operator, which was proved in the generality we need by Zorin-Kranich in \cite{zk-var-trunc}. In the second step, we replace the smooth cutoff functions by indicator functions,
using a standard approximation argument. Finally, we extend in the third step the range of exponents to $p \in (1, \infty)$ using a sparse domination result of Lorist \cite{Lorist2021}. To obtain Theorems \ref{t:avg} and \ref{t:sing}, it remains then to check that homogeneous Lie groups and their Leibman polynomials satisfy the assumptions of Theorem \ref{t:metric}. We verify this in Section \ref{s:lie}. 
Finally, we prove Corollary \ref{c:unitary} in Section \ref{s:quadratic}.    

\subsection*{Acknowledgements.}  A.J. and C.T. were funded by the Deutsche Forschungsgemeinschaft (DFG, German Research Foundation) under Germany's Excellence Strategy -- EXC-2047 -- 390685813. C.T. also  acknowledges funding by DFG CRC 1720 -- 539309657, and A.J. by DFG Heisenberg Grant -- 547294463. We are grateful to the anonymous referee for many helpful suggestions and corrections.

\section{Proof of the variational Wiener--Wintner theorem with smooth cutoffs}
\label{s:SmoTrun}

In this section, we prove Proposition \ref{p:SmoTrun} below, which is a variant of Theorem  \ref{t:metric} with a smaller range of exponents and smooth cutoff functions. 

\subsection{Smooth cutoff functions}
 For any function $\phi: [0,\infty) \to \mathbb{C}$ we denote
\[
    \phi_t(x) = t^{-1} \phi(t^{-1}x).
\]
For $\delta > 0$, we set 
\begin{equation}\label{e:zeta}
    \zeta^\delta(x) = \log(1 + \delta)^{-1} \mathbf{1}_{[1,1 + \delta]}(x). 
\end{equation}
This is chosen and normalized so that
\begin{equation}
    \supp \zeta^{\delta} \subset [1, 1+\delta], \label{e:zeta_supp}
\end{equation}
and
\begin{equation}
    \int_0^\infty (\zeta^{\delta})_t(x) \, \mathrm{d}t = 1, \qquad x \in (0,\infty).\label{e:zeta_int}
\end{equation}
Define then the $\delta$-smooth cutoff
\[
    Z_{a,b}^{\delta}(x) = \int_a^b (\zeta^\delta)_t(x) \, \mathrm{d}t.
\]
It has the property that
\begin{equation}\label{e:Z_prop}
    \mathbf{1}_{[(1 + \delta)a, b]} \le  Z_{a,b}^{\delta} \le \mathbf{1}_{[a, (1 + \delta)b]}.
\end{equation}

\subsection{Main proposition}
In what follows, we fix a $\AD$-dimensional metric measure space $(X, \rho, \mu)$ together with a $\varepsilon$-cancellative compatible collection $\mathcal{Q}$. All constants below are allowed to depend on that data.

Using $Z_{a,b}^\delta$ as our cutoff function rather than $\mathbf{1}_{[a,b]}$ yields the $\delta$-truncated modulated singular integrals
\begin{equation}
    \label{e:trsdefdelta}
    S_{a,b}^{\delta}(K, Q, f)(x) = \int Z^\delta_{a,b}(\rho(x,y)) K(x,y) f(y) e(Q(y)) \, \mathrm{d}\mu(y)
\end{equation}
and the $\delta$-truncated modulated averages
\begin{equation}
    \label{e:tradefdelta}
    A_R^{\delta}(Q, f)(x) = R^{-\AD} \int Z^{\delta}_{0, R}(\rho(x,y)) f(y) e(Q(y)) \, \mathrm{d}\mu(y).
\end{equation} 
We further denote $S_{a,\infty}^\delta = S_a^\delta$ and consistently with the notation introduced in \eqref{e:trsdef} and \eqref{e:tradef}
\[
    S_{a,b} := S_b - S_a, \qquad\qquad A_{a,b} := A_b - A_a\,,
\]
and similarly for $S^\delta$ and $A^\delta$. 

\begin{prop}
    \label{p:SmoTrun}
    Let $K$ be an $\alpha$-kernel on $X$ satisfying the cancellation condition \eqref{e:CZcancellation} and $p \in (1,2)$. For all $0 < \delta \le  1$ and $f \in L^p(X)$, we have
    \begin{equation}
        \label{e:SmoTrun1}
        \left\|\sup_{Q \in \mathcal{Q}} \left\| S^{\delta}_{u}(K, Q, f)(x)\right\|_{V^r_u}\right\|_p
        \leq C \delta^{-\alpha} \|f\|_p\,,
    \end{equation}
    and
    \begin{equation}
        \label{e:SmoTrun2}
        \left\|\sup_{Q \in \mathcal{Q}} \left\| A_u^\delta(Q, f)(x) \right\|_{V^r_u}\right\|_p
        \leq C \delta^{-\alpha} \|f\|_p\,.
    \end{equation}
\end{prop}

\subsection{The generalized polynomial Carleson theorem}

    We will deduce Proposition \ref{p:SmoTrun} from the generalized Carleson theorem proved in \cite{beckeretal}, which we now state. 
    
    For $0<\alpha\leq 1$, a one-sided $\alpha$-kernel $K$ on $X$ is a measurable function
    \begin{equation*}
        K \colon X \times X \to \mathbb{C}
    \end{equation*}
    such that for all $x, y', y \in X$ with $x \neq y$
    \begin{equation}\label{e:1skernel size}
        |K(x,y)| \le \rho(x,y)^{-\AD}\,,
    \end{equation}
    and if $2\rho(y,y') \le \rho(x,y)$, then
    \begin{equation}\label{e:1skernel y smooth}
        |K(x,y) - K(x,y')|
        \le
        \left(\frac{\rho(y,y')}{\rho(x,y)}\right)^{\alpha}
        \rho(x,y)^{-\AD}\,.
    \end{equation}
    Note the difference to \eqref{e:kernel y smooth}: For one-sided kernels smoothness is only assumed in the second variable.
    Associated to $K$ are the nontangential maximal operator
    \begin{equation}\label{e:def tang unm op}
        T_{*}f(x)
        := \sup_{R_1 < R_2} \sup_{\rho(x,x') < R_1}
        \Big| \int_{R_1 < \rho(x',y) < R_2}
        K(x',y) f(y) \, \mathrm{d}\mu(y) \Big|
    \end{equation}
    and the maximally modulated operator
    \begin{equation}\label{e:def main op}
        Tf(x)
        := \sup_{Q \in \mathcal{Q}} \sup_{0 < R_1 < R_2}
        \Big| \int_{R_1 < \rho(x,y) < R_2}
        K(x,y) f(y) e(Q(y)) \, \mathrm{d}\mu(y) \Big|\,.
    \end{equation}
    The main result of \cite{beckeretal} are restricted weak type $L^q$ bounds for $q \in (1,2]$ for the operator $T$ (under mildly more general assumptions). Specializing it to $\AD$-dimensional spaces, interpolation yields the following theorem. 

    \begin{theorem}[\cite{beckeretal}, Theorem 1.1]
    \label{t:gen Carleson}
        For all $1<q < 2$ and $0<\alpha\leq 1$, there exists a constant $C$ such that the following holds.
        Let $K$ be a one-sided $\alpha$-kernel on $X$. 
        Assume that for all $g \in L^2(X, \mu)$, \begin{equation}\label{e:nontanbound}
            \|T_{*}g\|_{2} \leq C \|g\|_2\,.
        \end{equation}
        Then for all $f \in L^q(X, \mu)$,
        \begin{equation}
            \label{e:strong}
            \|Tf\|_q \le C \|f\|_q.
        \end{equation}
    \end{theorem}

    \subsection{Nontangential maximal operators}
    In our application of Theorem \ref{t:gen Carleson}, the assumption \eqref{e:nontanbound} for the nontangential maximal operator $T_*$ follows from the next proposition which was proved by Zorin-Kranich in the generality we need here, see \cite[Theorem 1.3]{zk-var-trunc} for averages and see \cite[Theorem 1.8]{zk-var-trunc} for singular integrals.

\begin{prop}
    \label{p:NontVar}
    Let $K$ be an $\alpha$-kernel on $X$ satisfying the cancellation condition \eqref{e:CZcancellation} and let $r > 2$. For all $f \in L^2(X)$, we have
    $$
    \left\|\sup_{R > 0} \sup_{\rho(x,x') \leq R} \left\| S_{R, u}(K, 1, f)(x')\right\|_{V^r(u \in [R, \infty))}\right\|_2
    \leq C \|f\|_2\,
    $$
    and
    $$
    \left\|\sup_{R>0} \sup_{\rho(x,x') \leq R} \left\| A_u(1, f)(x') \right\|_{V^r(u \in (R, \infty))}\right\|_2 
    \leq C \|f\|_2\,.
    $$
    By convexity, the same holds for $S^\delta$ and $A^\delta$ for all $\delta > 0$.
\end{prop}

    \subsection{Proof of Proposition \ref{p:SmoTrun}: Singular integrals} 
    We start with estimate \eqref{e:SmoTrun1}.
    Our task is to estimate in $L^p$ by $C$ the quantity
    \[
         \delta^{\alpha} \cdot \sup_{J \in \mathbb{N}} \sup_{u_0 < \dotsb < u_J} \sup_{w \in \R^J, \|w\|_{\ell^{r'}} \le 1} \sup_{Q \in \mathcal{Q}} \Big| \sum_{j=1}^J w_j \cdot S^\delta_{u_{j-1}, u_j}(K, Q, f) \Big|\,.
    \]
    By monotone convergence, we can assume that each of the parameters in the suprema ranges over a finite set. Then we choose measurably a maximizer for the outer three suprema. Thus, it suffices to estimate, for all measurable functions $J: X \to \mathbb{N}$, $w: X \to \ell^{r'}$, and $u_0, \dotsc, u_{J(x)}: X \to (0,\infty)$, each with finite range, the quantity
    \begin{align}
        &\delta^{\alpha} \cdot \sup_{Q \in \mathcal{Q}} \Big| \sum_{j=1}^{J(x)} w_j(x) \cdot S^\delta_{u_{j-1}(x), u_j(x)}(K, Q, f)(x)\Big|\nonumber\\
        &= \sup_{Q \in \mathcal{Q}} \Big| \int_{u_0(x) \le \rho(x,y) \le (1+\delta)u_{J(x)}(x)} f(y) e(Q(y)) \mathbf{K}(x,y) \, \mathrm{d}\mu(y) \Big|\,, \label{e:lin_op}
    \end{align}
    where
    \begin{align*}
        \mathbf{K}(x,y) = \delta^{\alpha} \cdot K(x,y) \sum_{j=1}^{J(x)} w_j(x) \cdot Z_{u_{j-1}(x), u_j(x)}^{\delta}(\rho(x,y)) \, .
    \end{align*}
    This is a special case of Theorem \ref{t:gen Carleson}. It remains to check that $\mathbf{K}$ satisfies its assumptions:
    \begin{itemize}
        \item[a)] The nontangential maximal operator \eqref{e:def tang unm op} associated to $\mathbf{K}$ is bounded on $L^2$.
        \item[b)] Up to dividing by a constant, $\mathbf{K}$ is a one-sided $\beta$-kernel, where $\beta := \min\{1/2, \alpha\}$.
    \end{itemize}  
    Assumption a) is the content of Proposition \ref{p:NontVar}. Indeed, reversing the above steps, it holds
    \begin{align}
        T_* g(x) &= \sup_{R_1 < R_2} \sup_{\rho(x,x') < R_1} \Big| \int_{R_1 < \rho(x', y) < R_2} \mathbf{K}(x',y) g(y) \, \mathrm{d}\mu(y) \Big|\label{e:details}\\
        &= \delta^\alpha \sup_{R_1 < R_2} \sup_{\rho(x,x') < R_1} \Big| \sum w_j(x') S_{u_{j-1}(x') \vee R_1, u_j(x') \wedge R_2}^\delta (K, 1, g)(x') \Big|\nonumber\\
        &\quad + C\delta^\alpha M |g|(x)\,,\label{e:details2}
    \end{align}
    where the sum in the second line is over all $j$ such that $(u_{j-1}(x'), u_j(x'))$ intersects the interval $[R_1, R_2]$.
    Here we used that on $R_1(1 + \delta) < \rho(x', y) < R_2$, the integrand in \eqref{e:details} is exactly equal to what is obtained by expanding the next line, while the contribution of the small annulus $R_1 < \rho(x', y) < R_1(1 + \delta)$ is bounded by the Hardy--Littlewood maximal function of $g$ by \eqref{e:kernel size}.
    
    By Hölder's inequality and the definition of $w$, \eqref{e:details2} is bounded by
    \begin{align*}
        &\le \delta^\alpha \sup_{R > 0} \sup_{\rho(x,x') < R} \|S_{R, u}(K, 1, g)(x') \|_{V^r(u \in [R, \infty))} + C \delta^\alpha M|g|(x)\,,
    \end{align*}
    which is estimated in Proposition \ref{p:NontVar}.     
    
    The remainder of this section is devoted to assumption b). We have to show that for some constant $A$
    \begin{equation}
        \label{e:bk_upper}
        |\mathbf{K}(x,y)| \le \frac{A}{\rho(x,y)^\AD}
    \end{equation}
    and that for $2\rho(y,y') \le \rho(x,y)$
    \begin{equation}
        \label{e:bk_hol}
        |\mathbf{K}(x,y) - \mathbf{K}(x,y')| \le \Big( \frac{\rho(y,y')}{\rho(x,y)}\Big)^\beta \frac{A}{\rho(x,y)^\AD}\,.
    \end{equation}
    We write
    \begin{equation}
        \label{e:K_fact}
        \mathbf{K}(x,y) = K(x,y)\cdot \Gamma_x(\rho(x,y))
    \end{equation}
    where
    \begin{align*}
        \Gamma_x(s) = \delta^\alpha \cdot \int (\zeta^\delta)_t(s) \sum_{j=1}^{J(x)} w_j(x) \cdot  \mathbf{1}_{[u_{j-1}(x), u_j(x)]}(t) \, \mathrm{d}t\,.
    \end{align*}
    Since $K$ satisfies the two estimates \eqref{e:kernel size} and \eqref{e:kernel y smooth}, the estimates \eqref{e:bk_upper} and \eqref{e:bk_hol} for $\mathbf{K}$ with Hölder exponent $\beta = \min\{1/2,\alpha\}$ follow at once from the next lemma and \eqref{e:K_fact}.

    \begin{lemma}
        \label{l:gamma}
        It holds that
        \begin{equation}
            \label{e:g_upper}
            |\Gamma_x(s)| \le 1
        \end{equation}
        and, for all $s/2 \le s' \le s$ and $\gamma \le \min\{1/2, \alpha\}$, it holds that
        \begin{equation}
            \label{e:g_hol}
            |\Gamma_x(s) - \Gamma_x(s')| \le 8 \Big(\frac{s - s'}{s'}\Big)^\gamma.
        \end{equation}
    \end{lemma}

    \begin{proof}
    We consider $x, \delta$ fixed and drop them from the notation. Both estimates \eqref{e:g_upper} and \eqref{e:g_hol} scale correctly, so that we may assume $s = 1$ and $s' \in [1/2, 1)$. Since $\|w\|_{\ell^{r'}} \le 1$,
    it holds in particular for all $j$ that
    \[  
        |w_j| \le 1.
    \]
    This along with \eqref{e:zeta}, \eqref{e:zeta_int} implies \eqref{e:g_upper}. To show \eqref{e:g_hol}, we split the indices $j$ into
    \[
        \mathcal{I}_1 = \{j \in \{1, \dotsc, J\} \ : \ u_{j-1} > s' u_{j}\}
    \]
    and 
    \[
        \mathcal{I}_2 = \{j \in \{1, \dotsc, J\} \ : \ u_{j-1}  \le s' u_{j}\}.
    \]
    Using the support assumption \eqref{e:zeta_supp} and $s' \in [1/2,1)$, we have (recall that $\zeta_t = \zeta_t^\delta$)
    \begin{align*}
        &\delta^{-\alpha}|\Gamma(1) - \Gamma(s')|\\
        &= \Big| \int_{1/(2+2\delta)}^{1} \zeta_t(1) \sum_{j \in \mathcal{I}_1 \cup \mathcal{I}_2} w_j \mathbf{1}_{[u_{j-1}, u_j]}(t) - \zeta_t(s') \sum_{j \in \mathcal{I}_1 \cup \mathcal{I}_2} w_j \mathbf{1}_{[u_{j-1}, u_j]}(t)\, \mathrm{d}t \Big|.
    \end{align*}
    We apply the triangle inequality, and a change of variables $t \mapsto \frac{1}{s'}t $ in the $\Gamma(s')$ integral for the $j \in \mathcal{I}_2$ summands to bound this by
    \begin{align}
        &\quad \int_{1/(1+\delta)}^{1} |\zeta_t(1)| \Big|\sum_{j \in \mathcal{I}_1} w_j \mathbf{1}_{[u_{j-1}, u_j]}(t)\Big| \, \mathrm{d}t\label{e:I11}\\
        &+\int_{1/(1+\delta)}^{1} |\zeta_t(1)| \Big|\sum_{j \in \mathcal{I}_1} w_j \mathbf{1}_{[u_{j-1}, u_j]}(\frac{1}{s'}t)\Big| \, \mathrm{d}t\label{e:I12}\\
        &+ \int_{1/(1+\delta)}^{1} |\zeta_t(1)| \Big|\sum_{j \in \mathcal{I}_2} w_j \mathbf{1}_{[u_{j-1}, u_j]}(t) -  \sum_{j \in \mathcal{I}_2} w_j \mathbf{1}_{[u_{j-1}, u_j]}(\frac{1}{s'}t )\Big|\, \mathrm{d}t \label{e:I2}.
    \end{align}
    Using \eqref{e:zeta_int} and Jensen's inequality we bound \eqref{e:I11} by 
    \[
        \Big(\int_{1/(1+\delta)}^1 |\zeta_t(1)| \Big|\sum_{j \in \mathcal{I}_1} w_j \mathbf{1}_{[u_{j-1}, u_j]}(t)\Big|^{\frac{1}{\gamma}} \, \mathrm{d}t \Big)^\gamma.
    \]
    Note that $\zeta_t(1) \le 2 \delta^{-1} t^{-1}$ by \eqref{e:zeta}. This bounds the integral in the previous line by 
    \begin{align*}
        2\delta^{-1} \int \sum_{j \in \mathcal{I}_1} \mathbf{1}_{[u_{j-1}, u_j]}(t) |w_j|^{\frac{1}{\gamma}} \, \frac{\mathrm{d}t}{t} &= \sum_{j \in \mathcal{I}_1} |w_j|^{\frac{1}{\gamma}} \log\Big( \frac{u_j}{u_{j-1}} \Big)\\
        &\le \|w_j\|_{\frac{1}{\gamma}}^{\frac{1}{\gamma}} \log\Big(\frac{1}{s'} \Big)\,.
    \end{align*}
    By assumption, $\gamma \le 1/2$ and so 
    \[
        \|w_j\|_{\frac{1}{\gamma}} \le \|w_j\|_{r'} \le 1.
    \]
    Combining the above with the estimate $\log(x) \le 1 + x$ for $x \ge 1$ bounds \eqref{e:I11} by 
    \begin{equation}
        \label{e:zeta_ub}
        2 \delta^{-\gamma} \Big(\frac{1 - s'}{s'}\Big)^\gamma \le 2 \delta^{-\alpha} \Big(\frac{1 - s'}{s'}\Big)^\gamma.
    \end{equation}
    The same argument applies to \eqref{e:I12}.
    It remains to deal with the $\mathcal{I}_2$ term \eqref{e:I2}. We  notice that for $j \in \mathcal{I}_2$ 
    \begin{equation}
        \label{e:tel}
        \big|\mathbf{1}_{[u_{j-1}, u_j]}(t) - \mathbf{1}_{[u_{j-1}, u_j]}(\frac{1}{s'}t)\big| = \mathbf{1}_{[s'u_{j-1}, u_{j-1}]}(t) + \mathbf{1}_{[s'u_j, u_j]}(t).
    \end{equation}
    Using this in \eqref{e:I2}, and subsequently doing the same Jensen's inequality and computation as for term \eqref{e:I11} also estimates \eqref{e:I2} by \eqref{e:zeta_ub}. This completes the proof. 
    \end{proof}

    \subsection{Proof of Proposition \ref{p:SmoTrun}: Averages} 
    We turn to the second inequality \eqref{e:SmoTrun2}. Following the same linearization procedure as before, our task is to estimate in $L^p$ by a constant the function
    \begin{align*}
        &\delta^{\alpha} \cdot \sup_{Q \in \mathcal{Q}} \Big| \sum_{j=1}^{J(x)} w_j(x) \cdot (A^\delta_{u_j(x)}(Q, f)(x) - A^\delta_{u_{j-1}(x)}(Q, f)(x)) \Big|\\
        &= \sup_{Q \in \mathcal{Q}} \Big|\int_{\rho(x,y) \le (1+\delta)u_{J(x)}(x)} f(y) e(Q(y)) \mathbf{A}(x,y) \, \mathrm{d}\mu(y)\Big|\,,
    \end{align*}
    where
    \begin{equation}
        \label{e:AL}
        \mathbf{A}(x,y) = \Lambda_x(\rho(x,y)),
    \end{equation}
    \[
        \Lambda_x(s) = \delta^{\alpha} \cdot \sum_{j=1}^{J(x)} w_j(x)\cdot [ u_{j}(x)^{-\AD} Z_{0,u_j(x)}^{\delta}(s) -  u_{j-1}(x)^{-\AD} Z_{0,u_{j-1}(x)}^{\delta}(s)]\,.
    \]
    As before, this is a special case of Theorem \ref{t:gen Carleson} once we verify: 
    \begin{itemize}
        \item[a)] The boundedness of the nontangential maximal operator \eqref{e:nontanbound} associated to $\mathbf{A}$ on $L^2$.
        \item[b)] Up to dividing by a constant, $\mathbf{A}$ is a one-sided $\beta$-kernel, where $\beta = \min\{1/2, \alpha\}$.
    \end{itemize}  
    Part a) is again the content of Proposition \ref{p:NontVar}, after undoing the linearizations similarly as in the singular integral case. We skip the details here. Part b) follows from \eqref{e:AL} and the next lemma. 
    \begin{lemma}
        For all $x$ and $s$
        \begin{equation}
            \label{e:ba_upper}
            |\Lambda_x(s)| \le 2^{\AD + 1} s^{-\AD}
        \end{equation}
        and for all $s/2 \le s' \le s$ and $\gamma \le 1/2$
        \begin{equation}
            \label{e:ba_hol}
            |\Lambda_x(s) - \Lambda_x(s')| \le 2^{2\AD + 2} \Big( \frac{s - s'}{s'} \Big)^\gamma  s^{-\AD}.
        \end{equation}
    \end{lemma}
    
    \begin{proof}
        We fix again $x, \delta$ and drop them from the notation.
        Define the function
        \[
            \varphi(t) := -\AD Z^\delta_{0,1}(t) + \zeta^\delta(t),
        \]
        so that 
        \begin{align*}
            t^{\AD} \frac{\mathrm{d}}{\mathrm{d}t}[t^{-\AD} Z^\delta_{0, t}(s)] = -\AD  t^{-1} Z^\delta_{0,t} (s) + \zeta^\delta_t(s) = \varphi_t(s).
        \end{align*}
        By the fundamental theorem of calculus,
        \[
            \Lambda(s) = \delta^\alpha \cdot \int_0^\infty t^{-\AD} \varphi_t(s) \sum_{j=1}^J w_j \cdot \mathbf{1}_{[u_{j-1}, u_j]}(t) \, \mathrm{d}t.
        \]
        The estimates \eqref{e:ba_upper} and \eqref{e:ba_hol} scale correctly in $s$, so we may assume $s = 1$ and $s' \in [1/2, 1)$.
        
        For \eqref{e:ba_upper}, we bound $|w_j| \le 1$ and use that $\varphi$ is supported in $[0,1+\delta]$ to obtain the estimate
        \begin{equation*}
            |\Lambda_x(1)| \le \delta^\alpha \cdot \int_{1/(1+\delta)}^\infty t^{-\AD} |\varphi_t(1)| \, \mathrm{d}t = \delta^\alpha \cdot \int_0^{1+\delta} t^{\AD - 1} |\varphi(t)| \, \mathrm{d}t\,.
        \end{equation*}
        Using \eqref{e:zeta_int} and \eqref{e:Z_prop}, this integral is bounded by
        \begin{equation}\label{e:int_est}
            (1 + \delta)^\AD \int \zeta^\delta_t(1) \, \mathrm{d}t + \int_0^{1+\delta} \AD \cdot t^{\AD - 1} \, \mathrm{d}t = 2 (1 + \delta)^\AD \le 2^{\AD + 1}, 
        \end{equation}
        which proves \eqref{e:ba_upper}.
        For \eqref{e:ba_hol}, we write as in the proof of Lemma \ref{l:gamma} 
        \begin{align*}
            &\delta^{-\alpha}\cdot |\Lambda(1) - \Lambda(s')|\\
            &\le \int_{1/(1+\delta)}^\infty t^{-\AD}|\varphi_t(1)| \Big| \sum_{j \in \mathcal{I}_1} w_j \mathbf{1}_{[u_{j-1},u_j]}(t) \Big| \, \mathrm{d}t \\
            &+ \int_{1/(1+\delta)}^\infty t^{-\AD}|\varphi_t(1)| \Big| \sum_{j \in \mathcal{I}_1} w_j \mathbf{1}_{[u_{j-1},u_j]}(\frac{1}{s'} t) \Big| \, \mathrm{d}t\\
            &+ \int_{1/(1+\delta)}^\infty t^{-\AD}|\varphi_t(1)| \Big| \sum_{j \in \mathcal{I}_2} w_j \mathbf{1}_{[u_{j-1},u_j]}(t)  - \sum_{j \in \mathcal{I}_2} w_j \mathbf{1}_{[u_{j-1},u_j]}(\frac{1}{s'}t) \Big| \, \mathrm{d}t
        \end{align*}
        The computation \eqref{e:int_est} shows that we can apply Jensen's inequality to each of these three terms as in the proof of Lemma \ref{l:gamma}, up to losing a factor $2^{\AD + 1}$. Combining this with the fact that 
        \[
            |t^{-\AD}\varphi_t(1)| \le 2^{\AD + 1} \delta^{-1}
        \]
        and then the same estimate as in the proof of Lemma \ref{l:gamma} for the result of Jensen's inequality completes the proof. 
    \end{proof}

\section{Proof of the variational Wiener--Wintner theorem for \texorpdfstring{$p \in (1,2)$}{p in (1,2)}}
\label{s:ShTruformSmoTru}

In this section we replace the smooth cutoff functions in Proposition \ref{p:SmoTrun} by indicator functions. This uses a standard approximation argument, which previously appeared for example in \cite{DOP}.

\subsection{Reduction to variation along short sequences}

We start with a dyadic pigeonholing lemma that allows us to reduce to sequences of controlled length in the definition of the variation norms. Set 
\begin{equation}
    \label{e:epsdef}
    r_- = 1 + \frac{r}{2}, 
\end{equation}
and let $\varepsilon$ be a small positive number with 
\begin{equation}
    0 < \varepsilon \le 10^{-1}\big(\frac{1}{r_-} - \frac{1}{r}\big).
\end{equation}
It will be chosen sufficiently small depending on the other parameters at the end of the proof. 

\begin{lemma}
    For all $U: (0, \infty) \times X \to \mathbb{C}$ there exists $J_0 \ge 2$ such that
    \begin{align}\label{e:epsilonvar}
        &\big\|\|U(u,x)\|_{V^r_u}\big\|_{p}\nonumber\\
        &\le C(r,\varepsilon) J_0^{-\varepsilon} \Big\|\sup_{u_0<\dots <u_{J_0}} \Big(\sum_{j=1}^{J_0} |U(u_j, x)-U(u_{j-1},x)|^{r_-}\Big)^{\frac{1}{r_-}}\Big\|_p.
    \end{align}
\end{lemma}
\begin{proof}
For $k\in \N$, let $\lambda_k(x)$ be the supremum of all real numbers $\lambda$ such that there exists a sequence
\[u_1 < v_1 \leq u_2 < v_2 \leq \dotsb \leq u_{2^k} < v_{2^k}\]
with 
$$|U(u_j,x) - U(v_j,x)| \geq \lambda.$$ 
In any sequence $0<u_0<\dots<u_J$,
sort the differences $$ |U(u_j,x)-U(u_{j-1}, x)|$$ for $0<j\le J$
in decreasing order. Then for $2^k\le J$, 
the $2^k$-th difference is  at most $\lambda_k(x)$,
and so are all differences numbered between $2^k$ and $2^{k+1}-1$.
Thus
\begin{equation}\label{e:sortsum}
\sum_{j=1}^J|U(u_j,x)-U(u_{j-1}, x)|^{{r}}
\le \sum_{k=0}^\infty 2^k \lambda_k(x)^r. 
\end{equation}
It follows that the left-hand-side of \eqref{e:epsilonvar}
is bounded by
\[
    \Big\|\Big(\sum_{k=0}^\infty 2^k \lambda_k(x)^r\Big)^{\frac{1}{r}}\Big\|_p \le \Big\|\sum_{k=0}^\infty 2^{\frac{k}{r}}
    \lambda_k(x)\Big\|_p \le \sum_{k=0}^\infty 2^{-2\varepsilon k} \Big\| 2^{\frac{k}{r_-}}\lambda_k(x)\Big\|_p 
\]
where we use \eqref{e:epsdef}. Summing a geometric series, the previous is bounded by
\begin{equation}
    \label{e:epsilonk}
    C \sup_{k\in \N} 2^{-\varepsilon k} \Big\|  2^{\frac{k}{r_-}}\lambda_k(x)\Big\|_p . 
\end{equation}
Using the definition of $\lambda_k$, we can estimate \eqref{e:epsilonk}
by the right hand side of \eqref{e:epsilonvar} for $J_0=2^{k+1}$, where $k$ is a value that
attains the supremum in \eqref{e:epsilonk} up to possibly a factor of size 
at most $2$. 
\end{proof}

\subsection{Singular integrals}
We prove the bound \eqref{e:thm ShTru1}. 
Let $f$ be a function with $\|f\|_p = 1$.
By the monotone convergence theorem, we may restrict the supremum in $Q$ in \eqref{e:thm ShTru1} to a finite set. Then there exists a measurable function $Q: X \to \mathcal{Q}$ selecting a maximizer, and we denote
\begin{equation}\label{e:casetr}
    S(u,x) := S_u(K, Q(x), f)(x)\,, \qquad S^\delta(u,x) := S_u^\delta(K, Q(x), f)(x).
\end{equation}
Our task is to show
\begin{equation}
    \label{e:g bound}
    \|\|S(u,x)\|_{V^r_u}\|_{p}  \leq C\,. 
\end{equation}
We will compare the truncated singular integral $S(u,x)$ with its smooth version $S^\delta(u,x)$. Let
\[
    s = \frac{1 + p}{2}.
\]
\begin{lemma}
    \label{l:sing_comp}
    For all $0 < u_0 < u_1$ and $0 < \delta < 1$
    \begin{equation}
        \label{e:comp_tru}
        |S(u_0, x) - S(u_1,x) - S^\delta(u_0,x) + S^\delta(u_1,x)| \le C \delta^{1 -\frac{1}{s}} (M|f|^{s})^{\frac{1}{s}}(x). 
    \end{equation}
\end{lemma}

\begin{proof}
    We estimate the left hand side using the triangle inequality and \eqref{e:Z_prop} by 
    \[
        \sum_{j=0,1} \int_{u_j < \rho(x,y) < (1 + \delta)u_j} |K(x,y)| |f(y)| \, \mathrm{d}\mu(y).
    \]
    By \eqref{e:kernel size} and \eqref{e:ADregular}, this is at most
    \[
        C \sum_{j=0,1} u_j^{-\AD} \int_{u_j < \rho(x,y) < (1 + \delta)u_j} |f(y)| \, \mathrm{d}\mu(y).
    \]
    We estimate this using Hölder and \eqref{e:ADregular} by 
    \[
        C  \sum_{j=0}^1 [(1 + \delta)^\AD - 1]^{1 - \frac{1}{s}} (M |f|^s)^{\frac{1}{s}}(x) \le C \delta^{1 - \frac{1}{s}} (M |f|^s)^{\frac{1}{s}}(x). \qedhere
    \]
\end{proof}

To complete the proof of \eqref{e:g bound}, note that by Equation \eqref{e:epsilonvar}, it suffices to estimate
\[
    J_0^{-\varepsilon} \Big\|\sup_{u_0<\dots <u_{J_0}} (\sum_{j=1}^{J_0} |S(u_j, x)-S(u_{j-1},x)|^{r_-})^{\frac{1}{r_-}}\Big\|_p\,.
\]
Using the triangle inequality and Lemma \ref{l:sing_comp} for each term we estimate
\begin{align*}
    &\le J_0^{-\varepsilon} \Big\|\sup_{u_0<\dots <u_{J_0}} (\sum_{j=1}^{J_0} |S^{\delta}(u_j, x)-S^{\delta}(u_{j-1},x)|^{r_-})^{\frac{1}{r_-}}\Big\|_p\\
    &\quad + C \, J_0^{-\varepsilon + \frac{1}{r_-}} \delta^{1 - \frac{1}{s}} \|(M |f|^s)^\frac{1}{s}\|_p\,.
\end{align*}
The $\alpha$-kernel $K$ is also a $\beta$-kernel for all $\beta \le \alpha$. 
By equation \eqref{e:SmoTrun1} of Proposition \ref{p:SmoTrun} and the boundedness of the Hardy--Littlewood maximal function, the previous is hence for every $\beta \le \alpha$ further bounded by
\begin{equation}
    \label{e:optimize}
    C(\beta)[J_0^{-\varepsilon} \delta^{-\beta} + J_0^{-\varepsilon + \frac{1}{r_-}} \delta^{1 - \frac{1}{s}}].
\end{equation}
We choose the parameters as 
\[
    \delta = J_0^{-\nu}, \qquad \beta = \frac{\varepsilon}{\nu}, \qquad  \nu = \Big(-\varepsilon + \frac{1}{r_-}\Big)\frac{s}{s-1},
\]
and we choose $\varepsilon$ sufficiently small so that $\beta \le \alpha$. Note that $J_0 \ge 2$ and $\nu > 0$, hence this value of $\delta$ was admissible when applying \eqref{e:comp_tru}. 
Then \eqref{e:optimize} is bounded by $C$, which completes the proof.
 
\subsection{Averages}
The proof of the estimate  \eqref{e:thm ShTru2} in Theorem \ref{t:metric} is very similar to the proof of estimate \eqref{e:thm ShTru1}, so we only indicate the difference. We now compare the variations of the averages
\begin{equation*}
    A(u,x) := A_u(Q(x), f)(x)\,, \qquad A^\delta(u,x) := A_u^\delta(Q(x), f)(x),
\end{equation*}
and the comparison Lemma \ref{l:sing_comp} has to be replaced by the following. 

\begin{lemma}
    \label{l:avg_comp}
    For all $0 < u_0 < u_1$
    \begin{equation*}
        |A(u_0, x) - A(u_1,x) - A^\delta(u_0,x) + A^\delta(u_1,x)| \le C \delta^{1 -\frac{1}{s}} (M |f|^s)^\frac{1}{s}(x).
    \end{equation*}
\end{lemma}

\begin{proof}
    We directly estimate the left hand side by
    \[
        C \sum_{j=0,1} u_j^{-\AD} \int_{u_j < \rho(x,y) < (1 + \delta)u_j} |f(y)| \, \mathrm{d}\mu(y),
    \]
    from which the claim follows exactly as in the proof of Lemma \ref{l:sing_comp}.
\end{proof}

\section{Sparse bounds: Proof of Theorem \ref{t:metric} for \texorpdfstring{$p \in (1,\infty)$}{p in (1,infty)}}
\label{s:sparse}

Here we complete the proof of Theorem \ref{t:metric} by extending the range of exponents $p \in (1,2)$ proved in the previous section to $(1, \infty)$. 

\subsection{Sparse bounds on metric measure spaces}

We will apply the following special case of a result of Lorist \cite[Corollary 1.2]{Lorist2021}. 

\begin{theorem}\label{t:lorist}
    Let $Y$ be a Banach space. Let $T: L^{3/2}(X) \to L^{3/2}(X, Y)$ be a bounded linear operator. Define 
    \begin{equation}\label{e:sparse_nont}
        \mathcal{M}_T f(x) = \sup_{x\in B} \sup_{x',x'' \in B} \|T(f \mathbf{1}_{X \setminus 3B})(x') - T(f \mathbf{1}_{X \setminus 3B})(x'')\|_Y,
    \end{equation}
    where the supremum is over all balls $B$ in $X$. If $\mathcal{M}_T$ is bounded as an operator $L^{3/2}(X)\to L^{3/2}(X)$, then $T$ is bounded as an operator 
    \[
        L^p(X) \to L^p(X, Y)
    \]
    for all $p \in (3/2, \infty)$.
\end{theorem}

We apply this with the following choice of $Y$ and $T$. Let $Y$ be the Banach space of functions $G: \mathcal{Q} \times (0, \infty) \to \mathbb{C}$ with the norm  
\[
    \| G \|_Y = \sup_{Q \in \mathcal{Q}} (|G(Q,1)| + \|G(Q, u)\|_{V^r_u}).
\]
Let $T_a$ and $T_b$ be the operators mapping into $Y$-valued functions on $X$
\[
    T_af = A_t(Q, f), \qquad\qquad
    T_bf = S_t(K, Q, f)\,.
\]
They are bounded as operators $L^{3/2}(X) \to L^{3/2}(X, Y)$, as proved in Section \ref{s:ShTruformSmoTru}. By Theorem \ref{t:lorist} it remains only to show $L^{3/2}(X)$ boundedness of $\mathcal{M}_{T_a}$ and $\mathcal{M}_{T_b}$ to establish Theorem \ref{t:metric} for the full claimed range of exponents. It follows from the next proposition.

\begin{prop}\label{p:sparse}
    There exists a constant $C$ such that
    \[
        \mathcal{M}_{T_a}f + \mathcal{M}_{T_b}f \le C (M|f|^{4/3})^{3/4},
    \]
    where $M$ is the Hardy--Littlewood maximal function.
\end{prop}

\begin{proof}
    We may take $x'' = x$ in \eqref{e:sparse_nont}, at the cost of a factor of at most $2$. We may also assume, by scaling invariance of all assumptions, that $B$ is a ball of radius $1$. 
        
   We start with $\mathcal{M}_{T_b}$. Our task is to estimate for any $Q \in \mathcal{Q}$
    \begin{equation}\label{e:var}
        \|S_u(K, Q, f \mathbf{1}_{X \setminus B(x,3)})(x') - S_u(K, Q, f\mathbf{1}_{X \setminus B(x,3)})(x)\|_{V^r_u}.
    \end{equation}
    Since $r > 2$, there exists $J$ and a sequence
    \begin{equation}\label{e:partition}
        u_0 < \dotsb < u_J
    \end{equation}
    such that \eqref{e:var} is bounded by
    \begin{equation}\label{e:var2}
         2 \Big(\sum_{j = 1}^J |S_{u_{j-1}, u_j}(K, Q, f \mathbf{1}_{X \setminus 3B})(x') - S_{u_{j-1}, u_j} (K, Q, f \mathbf{1}_{X \setminus 3B})(x)|^{4/3}\Big)^{3/4}.
    \end{equation}
    Since $f\mathbf{1}_{X \setminus 3B}$ vanishes on the ball $B(x', 2)$, we can assume that $u_0 \ge 2$. 
    
    We distinguish normal and tiny intervals in the partition \eqref{e:partition}. 
    An interval $[u_{j-1}, u_j]$ is called tiny if 
    \[
        u_j \le u_{j-1} + u_{j-1}^\rho, \qquad \rho = \max\{1 - \frac{\alpha}{2}, \frac{7}{8}\}\,,
    \]
    and it is normal if it is not tiny. Let $[u_{j-1}, u_j]$ be a normal interval. Then 
    \begin{align*}
        &|S_{u_{j-1}, u_j}(K, Q, f \mathbf{1}_{X \setminus 3B})(x') - S_{u_{j-1}, u_j} (K, Q, f \mathbf{1}_{X \setminus 3B})(x)|\\
        &=  \Big\lvert \int_{u_{j-1} < \rho(x',y) < u_j} K(x', y) e(Q(y)) f(y) \mathbf{1}_{X \setminus 3B}(y) \, \mathrm{d}\mu(y)\\
        &\qquad- \int_{u_{j-1} < \rho(x,y) < u_j} K(x, y) e(Q(y)) f(y) \mathbf{1}_{X \setminus 3B}(y) \, \mathrm{d}\mu(y)\Big\rvert
    \end{align*}
    and by the triangle inequality
    \begin{align}
        &\le \int_{u_{j-1} < \rho(x,y) < u_j} |K(x,y) - K(x', y)| |f(y)| \, \mathrm{d}\mu(y)\label{e:lvar1}\\
        &+\int_{A_{j} \cup A_{j-1}} |K(x',y)| |f(y)| \, \mathrm{d}\mu(y)\label{e:lvar2}
    \end{align}
    where 
    \begin{equation}\label{e:ajdef}
        A_j = B(x, u_{j} + 2) \setminus B(x, u_{j} - 2).
    \end{equation}
    The term \eqref{e:lvar1} is, by \eqref{e:kernel y smooth}, at most
    \begin{equation}\label{e:normal1}
        \int_{u_{j-1} \le \rho(x,y) \le u_j} \Big( \frac{\rho(x,x')}{\rho(x,y)} \Big)^\alpha \rho(x,y)^{-\AD} |f(y)| \, \mathrm{d}\mu(y) \le C u_{j-1}^{-\alpha} M|f|(x)\,.
    \end{equation}
    For the term \eqref{e:lvar2} we use that by the $\AD$-dimensionality condition 
    \begin{equation}\label{e:smallAj}
        \mu(A_j) \le C \, u_j^{\AD - 1}.
    \end{equation}
    With Hölder's inequality and \eqref{e:kernel size}, it follows that the term \eqref{e:lvar2} is bounded by 
    \begin{equation}\label{e:normal2}
        C \sum_{i=0}^1  \Big(\frac{\mu(A_{j-i})}{u_{j-i}^\AD}\Big)^{1/4} \Big(\frac{1}{u_{j-i}^\AD} \int_{A_{j-i}} |f|^{4/3} \, \mathrm{d}\mu \Big)^{3/4} \le C \sum_{i=0}^1 u_{j-i}^{-1/4} (M|f|^{4/3})^{3/4}\,.
    \end{equation}
    There are at most $\min\{2^{k\alpha/2}, 2^{k/8}\}$ many normal intervals with $u_{j-1} \in [2^k, 2^{k+1}]$. Combined with \eqref{e:normal1} and \eqref{e:normal2}, it follows that the contribution of all normal intervals to \eqref{e:var2} is bounded by 
    \[
        C (M|f|^{4/3})^{3/4}\,.
    \]
    
    We turn to the tiny intervals. By $\AD$-dimensionality of $X$ and tininess
    \[
        \mu(\{y \ : \ u_{j-1} \le \rho(z, y) \le u_j\}) \le C (u_j - u_{j-1}) u_j^{\AD - 1} \le C u_{j-1}^{\AD + \rho - 1}\,.
    \]
    With Hölder, it follows that for $z \in \{x, x'\}$ and tiny $[u_{j-1}, u_j]$
    \begin{align*}
        &|S_{u_{j-1}, u_j}(K, Q, f \mathbf{1}_{X \setminus 3B})(z)|\\
        &\le C \Big(u_{j-1}^{\rho -1}\Big)^{1/4} \Big(\int_{u_{j-1} \le \rho(z, y) \le u_j} u_{j-1}^{-\AD} |f|^{4/3} \, \mathrm{d}\mu \Big)^{3/4}\,.
    \end{align*}
    Thus, the contribution of all tiny intervals to \eqref{e:var} is at most
    \begin{align*}
        &C \Big( \sum_{z \in \{x,x'\}} \sum_{\mathrm{tiny}} \int_{u_{j-1} \le \rho(z, y) \le u_j} u_{j-1}^{-\AD + (\rho - 1)/3} |f|^{4/3} \, \mathrm{d}\mu\Big)^{3/4} \\
        &\le C (M |f|^{4/3})^{3/4}(x)\,.
    \end{align*}
    Here we used that $u_{j-1} \sim u_j$, by tininess. This completes the proof for $\mathcal{M}_{T_b}$. 

    For $\mathcal{M}_{T_a}$ we proceed similarly. Now we have the simpler estimate
    \begin{align}
        &|A_{u_{j-1}, u_j}(Q, f \mathbf{1}_{X \setminus 3B})(x') - A_{u_{j-1}, u_j} (Q, f \mathbf{1}_{X \setminus 3B})(x)|\nonumber\\
        &\le \sum_{i = 0}^1 u_{j - i}^{-\AD} \int_{A_{j - i}} |f(y)| \, \mathrm{d}\mu\,,\label{e:lvar2_a}
    \end{align}
    where $A_j$ is defined in \eqref{e:ajdef}. This is the same as in the singular integral case for the term \eqref{e:lvar2}, so the contribution of the normal intervals can be controlled as shown there. 

    For tiny intervals, we use that 
    \begin{align*}
        &|A_{u_{j-1}, u_j}(f\mathbf{1}_{X \setminus 3B})(z)|\\
        &\le u_j^{-\AD} \int_{u_{j-1} \le \rho(z,y) \le  u_j} |f(y)| \, \mathrm{d}\mu + (u_{j-1}^{-\AD} - u_j^{-\AD}) \int_{\rho(z,y) \le u_{j-1}} |f(y)| \, \mathrm{d}\mu\,.
    \end{align*}
    For the first term, we argue again as in the singular integral case. The second summand is bounded by
    \[
        \AD \frac{u_j - u_{j-1}}{u_{j-1}} M|f|(x)\,.
    \]
    Taking now an $\ell^{4/3}$ sum in $j$ yields at most $M|f|(x)$ times the factor
    \begin{align*}
        &\quad \AD \Big(\sum_{\mathrm{tiny}} \Big(\frac{u_{j} - u_{j-1}}{u_{j-1}} \Big)^{4/3} \Big)^{3/4} \\
        &\le \AD \Big(\sum_{k \ge 1} 2^{k (\rho / 3 - 4/3)} \sum_{2^{k-1} \le u_{j-1} < 2^k, \mathrm{tiny}} |u_j - u_{j-1}| \Big)^{3/4}\\
        &\le \AD \Big(\sum_{k \ge 1} 2^{k (\rho / 3 - 4/3 + 1)} \Big)^{3/4}  \le C \,.
    \end{align*}
    This completes the proof.
\end{proof}

\section{Specification to homogeneous Lie groups}\label{s:lie}
 
Here we prove Corollary~\ref{c:hom Lie}, by verifying that homogeneous Lie groups satisfy the properties listed in Section \ref{ss:setup results}. 

Fix the homogeneous Lie group $G = (\R^n, \circ, \delta_\lambda)$ with group law $\circ$ and dilations $\delta_\lambda$ as in Section \ref{ss:lie}. Further fix a degree $d \ge 1$ and denote by $\mathcal{Q}$ the collection of unital, i.e. satisfying $Q(e) = 0$, measurable Leibman polynomials on $G$ of degree at most $d$. Note that Corollary~\ref{c:hom Lie} for unital polynomials is equivalent to the stated version with all polynomials, as adding a constant to the polynomial does not change the absolute value inside the suprema in \eqref{e:hom lie1} and \eqref{e:hom lie2}.
In what follows, all constants are allowed to depend on the group $G$ and on the degree $d$. The homogeneous dimension of $G$ is $\AD$.

\begin{lemma}\label{l: comp ball}There exists a constant $C_1$ such that for all $\lambda > 0$
    \[
        [-C_1^{-1}, C_1^{-1}]^n \subset \delta_{\lambda^{-1}} B(0, \lambda) \subset  [-C_1, C_1]^n.
    \]
\end{lemma}

\begin{proof}
    We may pick $C_1$ such that this holds for $\lambda  = 1$, but by homogeneity of the metric the expression in the middle does not depend on $\lambda$.
\end{proof}

Using Lemma \ref{l: comp ball}, we verify properties (1) to (5) of the metrics $d_B$. 

\begin{lemma}
    The collection $\mathcal{Q}$ forms a compatible collection of functions on $G$.
\end{lemma}

\begin{proof} 
    Property (1) holds for unital polynomials by definition with $x_0 = e$. Property (2) holds because a polynomial that is constant on any ball must be constant.

    We turn to property (3). By right invariance and scaling, we may assume $x_1=0$ and $d_{B_1}(f,g)=1$ and $R = 1$. Then $B_2 \subset B(0, 4)$. 
    Set $h=f-g$. By Lemma \ref{l: comp ball}, it suffices to show that there exists $C_2 > 0$ so that
    \begin{equation}\label{e:prop3proof}
        \sup_{\delta_{4}([-C_1, C_1]^n)} |h| \le C_2 \sup_{[-C_1^{-1}, C_1^{-1}]^n} |h|.
    \end{equation}
        Recall that $h$ is a classical polynomial on $\mathbb{R}^n$ of bounded degree $d_\R$ depending on $d$ and $G$. Inequality \eqref{e:prop3proof} then holds for some $C_2$ because all norms on the finite dimensional vector space of such polynomials are equivalent.

    We turn to property (4). We may again assume that $x_1 = 0$. Using a multivariate Lagrange interpolation formula (cf.~\cite[Theorem~3.1]{nicolaides1972class}), the coefficients of a polynomial can be determined linearly from its values on any cuboid. Hence, there exists a constant $C_3$ such that for every polynomial $p$ of degree $d_\R$ 
    \[
        2C_2 \sup_{[-C_1, C_1]^n} |p| \le  \sup_{\delta_{2C_3}( [-C_1^{-1}, C_1^{-1}]^n)} |p|\,.
    \]
    By Lemma \ref{l: comp ball} and dilation invariance, this implies
    \[
        2C_2 d_{B_1}(f,g) \le d_{B(0, 2C_3)}(f,g).
    \]
    Combining this with property (3) yields for all $x_2 \in B(0, C_3)$ that
    \[
        2 d_{B_1}(f,g) \le d_{B(x_2, C_3)}(f,g)\,,
    \]
    as required.

    For property (5) it suffices by Lemma \ref{l: comp ball} and dilation invariance to show that in the space $\mathcal{Q}$, every
    $L^\infty([-C_1^{-1}, C_1^{-1}]^n)$ ball of radius $2$ can be covered by at most $C_4$ many 
    $L^\infty([-C_1, C_1]^n)$
    balls of radius $1$. Both of these are just fixed norms on the finite dimensional vector space of polynomials of degree $d_\R$ on $\R^n$. All such norms are equivalent, and it is well known that they are all geometrically doubling.

    This completes the proof, with constant $C = \max\{C_2, C_3, C_4\}$.
\end{proof}

It remains to verify the cancellative condition. We deduce it from van der Corput's lemma for oscillatory integrals. A convenient form for our purposes was proved in \cite[Lemma~A.1]{zk-polynomial}. To adapt this lemma from the abelian group $\R^n$ to our context, we need the following.

\begin{lemma}\label{l:abel_holder}
    There exists a constant $C_5 > 0$ such that if $x,y \in B(0,1)$, then 
    \begin{equation}\label{e:abel_holder}
        \rho(x, x-y) \le C_5 \rho(e,y)^{1/\AD}\,,
    \end{equation}
    where $x - y$ denotes the usual abelian group law on $\R^n$.
\end{lemma}

\begin{proof}
    By Lemma \ref{l: comp ball},
    \begin{equation}\label{e:y}
        |y_i| \le C_1\rho(e,y), \qquad i = 1, \dotsc, n\,.
    \end{equation}
    Denote
    \[
        q(x,y) = x \circ (x - y)^{-1}.
    \]
    The function $q(x,y)$ is a polynomial in $x, y$ because $G$ is a nilpotent group. Therefore, there exists a constant $C_6$ such that 
    \begin{equation}\label{e:abel_holder2}
        \max_{i,j = 1, \dotsc, n} \sup_{x, y \in [-C_1, C_1]^n} |\frac{\partial}{\partial y_i} q_j(x,y)| \le C_6\,.
    \end{equation}
    Further, we have that $q(x,0) = 0$. Combining this with right invariance of $\rho$, Lemma \ref{l: comp ball}, \eqref{e:abel_holder2} and \eqref{e:y} yields
    \begin{align*}  
        \rho(x,x-y) &= \rho( q(x,y), e)\\
        &\le C_1 \max_{j=1,\dotsc, n} |q_j(x,y)|^{1/\AD}\\
        &\le nC_1 C_6 \max_{i = 1, \dotsc, n} |y_i|^{1/\AD}\\
        &\le n C_1^2 C_6 \ (\rho(e, y))^{1/\AD}\,.
    \end{align*}
    This completes the proof.
\end{proof}

\begin{lemma}
There exists $\varepsilon > 0$ such that $\mathcal{Q}$ forms an $\varepsilon$-cancellative collection on $G$. 
\end{lemma}

\begin{proof}
Note that on a homogeneous group, the cancellative condition \eqref{e:vdc cond} is consistent with group translations and dilations. Thus we may assume that $B$ is centered at $0$ and that $R = 1$.
Let $f,g \in \mathcal{Q}$. Recall that this implies that, after identifying $G$ with $\R^n$, they are polynomials of degree at most $d_\R$, for some $d_\R$ depending only on $G$ and $d$. Let $h = f - g$. By Lemma \ref{l: comp ball} and the support assumption on $\psi$,
\[
    \int_B e((f-g)(x)) \psi(x) \, \mathrm{d}\mu(x) = \int_{[-C_1, C_1]^n} e(h(x)) \psi(x) \, \mathrm{d}\mu(x).
\]
Since $h$ is a polynomial of degree at most $d_\R$ on $\R^n$, Lemma~A.1 of \cite{zk-polynomial} bounds the previous by
\[
     C_7 \, \sup_{y \in \eta [-C_1, C_1]^n} \int_{\R^n} |\psi(x) - \psi(x-y)| \, \mathrm{d}x, \qquad \eta = (1 + d_B(f,g))^{-1/d_\R}\,.
\]
Note that $x-y$ here denotes the abelian group law on $\R^n$. 
By Lemma \ref{l: comp ball}, and since both $\psi(x)$ and $\psi(x-y)$ are supported in sets of bounded measure, the previous is at most
\[
     C_8   \sup_{\rho(e,y) \le C_1^2 \eta} \sup_{\rho(e,x) \le 1} |\psi(x) - \psi(x- y)|\,.
\]
Without loss of generality, we may assume that $C_1^2 \eta \le 1$, otherwise the cancellative condition already follows from estimating $\psi$ pointwise by $\|\psi\|_{C^{0,1}}$. Then we  bound the previous using Lemma \ref{l:abel_holder} by
\[
    2C_8 \sup_{x \in B} \sup_{\rho(x,z) \le C_5 (C_1^2\eta)^{1/\AD}} |\psi(x) - \psi(z)| \le C_9 \eta^{1/\AD} \|\psi\|_{C^{0,1}}\,.
\]
This completes the proof with $\varepsilon = 1/(\AD d_\R)$. 
\end{proof}

\section{Polynomials into unitary groups: Corollary \ref{c:unitary}}

\label{s:quadratic}

Here we prove Corollary \ref{c:unitary}. It is a consequence of the following characterization of quadratic polynomials from nilpotent Lie groups into unitary groups, deduced from the main result of \cite{quadratic}.

\begin{prop}\label{p:abelian}
Let $G$ be a nilpotent, connected, simply connected Lie group, and let $U(n)$ be a unitary group. 
If $Q\colon G\to U(n)$ is a unital Leibman quadratic map, then $Q(G)$ is contained in a torus subgroup of $U(n)$.
\end{prop}

\begin{proof}
By \cite[Theorem~1.2]{quadratic}, there exists a unique homomorphism 
\[
    \psi\colon \mathrm{Pol}_2(G)\to U(n)
\]
such that $Q=\psi\circ \mathrm{quad}_G$, where 
$\mathrm{quad}_G\colon G\to \mathrm{Pol}_2(G)$ is the universal quadratic map.  
Thus it suffices to show that every homomorphism 
$\psi\colon \mathrm{Pol}_2(G)\to U(n)$ has abelian image.

By \cite[Theorem~1.3]{quadratic}, there is a short exact sequence
\[
0 \longrightarrow A:=\omega(G)\otimes_{\mathbb{Z}} G^{\mathrm{ab}}
   \longrightarrow \mathrm{Pol}_2(G)
   \longrightarrow G 
   \longrightarrow 0,
\]
where $\omega(G)$ is the augmentation ideal of $\mathbb{Z}[G]$ and where 
$A$ is abelian.  
Since $G$ is nilpotent, the extension shows that $\mathrm{Pol}_2(G)$ is a solvable group.

Next we verify divisibility.  
Because $G$ is connected and simply connected nilpotent, the exponential map 
$\exp\colon\mathfrak g\to G$ is a global diffeomorphism.  
Given $g=\exp(X)\in G$ and $m\ge1$, set $h=\exp(X/m)$.  
Then $h^m=\exp(m\cdot X/m)=g$, so $G$ is a divisible group.  
Moreover, $G^{\mathrm{ab}}\cong \mathbb{R}^k$ is a real vector space, hence divisible abelian.  
Since tensoring any abelian group with a $\mathbb{Q}$-vector space produces a
$\mathbb{Q}$-vector space, the group
\[
A=\omega(G)\otimes_{\mathbb{Z}} G^{\mathrm{ab}}
\]
is a divisible abelian group.

Let $\psi'\colon \mathrm{Pol}_2(G)\to F$ be any homomorphism into a finite group $F$.  
The restriction $\psi'|_A$ is trivial (no finite group contains nontrivial divisible subgroups),
so $\psi'$ factors through the quotient $\mathrm{Pol}_2(G)/A\cong G$.
But $G$ is divisible, hence admits no nontrivial homomorphisms to finite groups.
Thus $\psi'\equiv 1$.  
Consequently, every homomorphism $\mathrm{Pol}_2(G)\to$ (finite group) is trivial.

Now let $H=\psi(\mathrm{Pol}_2(G))$ and let $K=\overline{H}\subseteq U(n)$ be its closure.
Then $K$ is a compact, solvable Lie subgroup of $U(n)$, so its identity component
$K^\circ$ is a torus.  
Because $K$ is compact, the quotient $K/K^\circ$ is a finite group.  
The composite map
\[
\mathrm{Pol}_2(G)\xrightarrow{\psi} K \longrightarrow K/K^\circ
\]
is therefore a homomorphism into a finite group, hence trivial.
Thus $\psi(\mathrm{Pol}_2(G))\subseteq K^\circ$, which is abelian.

Hence $Q(G)=\psi(\mathrm{quad}_G(G))$ lies in a torus subgroup of $U(n)$.
\end{proof}

\begin{remark}\label{higher-degree}
For the above argument to extend to arbitrary unital Leibman polynomial maps 
$Q \colon G \to U(n)$ of degree at most $d$, it would be sufficient to know that, 
for every $d \ge 3$, the successive quotient 
\[
    \mathrm{Pol}_d(G) \big/ \mathrm{Pol}_{d-1}(G)
\]
is a divisible and solvable group.  
Indeed, under this hypothesis, one could repeat verbatim the proof for $\mathrm{Pol}_2(G)$:
every homomorphism $\mathrm{Pol}_d(G)\to$ (finite group) would vanish on the divisible kernel and on the divisible quotient, and hence would be trivial; the solvability of $\mathrm{Pol}_d(G)$ would then force the image of any homomorphism $\mathrm{Pol}_d(G)\to U(n)$ to lie in a torus and thus to be abelian.

At present, however, no explicit structural description of $\mathrm{Pol}_d(G)$ is known for $d\ge 3$ that is comparable to the formula for $\mathrm{Pol}_2(G)$ obtained in 
\cite[Theorem~1.3]{quadratic}, cf.~recent results of Alekseev and Thom \cite{at}.   
We only know that $\mathrm{Pol}_d(G)$ is characterized by the appropriate universal property (the higher-degree analogue of \cite[Theorem~1.2]{quadratic}), and this abstract description does not yet allow one to verify divisibility or solvability of the higher-degree layers.
\end{remark}

In order to apply Theorem \ref{t:avg} to Corollary \ref{c:unitary} we also need the following lifting result, which characterizes measurable Leibman polynomials with image in tori.

\begin{remark}\label{rem:lifting}
Let $G$ be a homogeneous Lie group, and let $T$ be a $k$-dimensional torus. Suppose that $P:G\to T$ is a measurable, and thus continuous \cite{auto-cont}, Leibman polynomial.

Choose an isomorphism $\psi:(S^1)^k\to T$. Since $G$ is simply connected, the map
$\psi^{-1} \circ P:G\to (S^1)^k$ admits a global continuous lift through the covering map
$$
    \mathbb R^k\to (S^1)^k,\qquad
    (x_1,\dots,x_k)\mapsto (e^{ix_1},\dots,e^{ix_k}).
$$
Thus there exist continuous functions $Q_1,\dots,Q_k:G\to \mathbb R$ such that
$$
    P(g)
    =
    \psi\bigl(e^{iQ_1(g)},\dots,e^{iQ_k(g)}\bigr).
$$
Moreover, $Q=(Q_1,\dots,Q_k)$ is itself an $\mathbb R^k$-valued Leibman polynomial, thus each $Q_j$ is a classical polynomial. 
Indeed, if $\deg P\leq d$, then
$$
    \Delta_{h_1}\cdots \Delta_{h_{d+1}}Q(g)
    \in 2\pi\mathbb Z^k.
$$
The right-hand side is a discrete lattice in $\mathbb R^k$, while the left-hand side is
continuous in $g,h_1,\dots,h_{d+1}$. Since $G$ is connected, this difference is constant,
and evaluating at $h_j=e_G$ gives zero. Thus
$$
    \Delta_{h_1}\cdots \Delta_{h_{d+1}}Q(g)=0,
$$
so each $Q_j$ is a continuous real-valued Leibman polynomial of degree at most $d$, hence a classical polynomial of degree at most $C(d)$.
\end{remark}

Corollary \ref{c:unitary} then follows by first applying Proposition \ref{p:abelian}, to deduce that up to a shift $Q$ has image in a torus. Then one applies Remark \ref{rem:lifting} to deduce that there exist real-valued polynomials $Q_1, \dotsc, Q_k$ such that
\[
    \phi(Q(g)) = \phi \circ \psi^{-1}(Q_1(g), \dotsc, Q_k(g))\,.
\]
Finally, one Fourier expands the function $\phi \circ \psi^{-1}$ on $(S^1)^k$. Note that $k \le n$, since $k$ is the dimension of a torus contained in a unitary group of dimension at most $n$. This implies that the $C^n$ function $\phi \circ \psi^{-1}$ is in the Wiener algebra, with Wiener algebra norm bounded only in terms of $n$ and $C$ from Corollary \ref{c:unitary}. Thus 
\[
    \phi(Q(g)) = \sum_{\xi \in \mathbb{Z}^k} c_\xi e(\xi_1 Q_1(g) + \dotsb + \xi_k Q_k(g))
\]
and there exists a constant $K = K(n, C)$ such that
\begin{equation}\label{e:Wiener}
    \sum_{\xi \in \mathbb{Z}^k} |c_\xi| \le K\,.
\end{equation}
Finally, one applies Theorem \ref{t:avg} with each of the polynomial phases $\xi_1 Q_1 + \dotsb + \xi_k Q_k$, and sums the resulting estimates using \eqref{e:Wiener}. This proves Corollary \ref{c:unitary}.

\appendix

\section{Transference}\label{transference}

We spell out the transference step from Corollary \ref{c:hom Lie} to
Theorem \ref{t:avg}.  In this generality the Calderón transference principle is also stated by Nevo \cite[Theorem~6.2]{nevo}. Although Nevo considers only maximal functions, the proof uses only sublinearity, the group 
$L^p$-estimate, and locality in the averaging parameter.  These
properties are unchanged when the supremum norm in the parameter is
replaced by the $r$-variation norm. 

Fix $R_0<\infty$.  For $F\in L^p(G,\mu)$, define
\[
\mathcal C_{R_0}F(a)
:=
\sup_Q
\left\|
\frac{1}{\mu(B(0,R))}
\int_{B(0,R)}
F(g\circ a)e(Q(g))\,d\mu(g)
\right\|_{V^r_{R \in (0, R_0]}},
\]
where here and below all suprema are over all measurable Leibman polynomials $Q$ of degree at most $d$.

We first note that Corollary \ref{c:hom Lie} implies, with $C$ independent of $R_0$,
\[
        \|\mathcal C_{R_0}F\|_{L^p(G,\mu)}
        \le C\|F\|_{L^p(G,\mu)},
\]
Indeed, since the metric on $G$ is right-invariant,
\[
        B(a,R)=B(0,R)\circ a \,.
\]
Moreover, since right translations
preserve Haar measure,  after the change of variables
$y=g\circ a$,
\[
\frac{1}{\mu(B(0,R))}
\int_{B(0,R)}F(g\circ a)e(Q(g))\,d\mu(g)
=
A_R(Q_a,F)(a),
\]
where
\[
        Q_a(y):=Q(y\circ a^{-1}).
\]
The function $Q_a$ is again a real-valued Leibman polynomial of degree
at most $d$.  Hence the preceding display is dominated by
the left-hand side of Corollary  \ref{c:hom Lie}.

Now let $T:G\times X\to X$ be a measure-preserving action on
$(X,\nu)$, and define
\[
\mathcal C^T_{R_0}f(x)
:=
\sup_Q
\left\|
\frac{1}{\mu(B(0,R))}
\int_{B(0,R)}
f(T^g x)e(Q(g))\,d\mu(g)
\right\|_{V^r_{R \in (0, R_0]}}.
\]Let
\[
        \Lambda_{R_0}:=\overline{B(0,R_0)}
        \qquad\text{and}\qquad
        \Omega_S:=B(0,S).
\]
For $x\in X$, define the truncated orbit function
\[
        F_{x,S}(a)
        :=
        \mathbf 1_{\Lambda_{R_0}\Omega_S}(a)\, f(T^a x),
        \qquad a\in G.
\]
If $a\in\Omega_S$, $g\in B(0,R)$, and $R\le R_0$, then
$g\circ a\in \Lambda_{R_0}\Omega_S$.  Therefore we have
\[
        \mathcal C^T_{R_0}f(T^a x)
        =
        \mathcal C_{R_0}F_{x,S}(a),
        \qquad a\in\Omega_S .
\]
Consequently,
\begin{align}
\mu(\Omega_S)\|\mathcal C^T_{R_0}f\|_{L^p(X,\nu)}^p
&=
\int_{\Omega_S}\int_X
|\mathcal C^T_{R_0}f(T^a x)|^p
\,d\nu(x)\,d\mu(a) \nonumber\\
&\le
\int_X\int_G
|\mathcal C_{R_0}F_{x,S}(a)|^p
\,d\mu(a)\,d\nu(x) \nonumber\\
&\le
C^p
\int_X\int_G |F_{x,S}(a)|^p
\,d\mu(a)\,d\nu(x) \nonumber\\
&=
C^p
\int_X\int_{\Lambda_{R_0}\Omega_S}
|f(T^a x)|^p
\,d\mu(a)\,d\nu(x)\nonumber \\
&=
C^p\mu(\Lambda_{R_0}\Omega_S)\|f\|_{L^p(X,\nu)}^p . \label{e:finite_rad}
\end{align}
Since $0\in\Lambda_{R_0}$, we have
$\Omega_S\subset\Lambda_{R_0}\Omega_S$.  On the other hand, by the
triangle inequality and right-invariance of the metric,
\[
        \Lambda_{R_0}\Omega_S\subset B(0,S+R_0).
\]
Hence, 
\[
1
\le
\frac{\mu(\Lambda_{R_0}\Omega_S)}{\mu(\Omega_S)}
\le
\frac{\mu(B(0,S+R_0))}{\mu(B(0,S))}
\longrightarrow 1
\]
as $S\to\infty$. Dividing \eqref{e:finite_rad} by $\mu(\Omega_S)$
and letting $S\to\infty$, we obtain
\[
        \|\mathcal C^T_{R_0}f\|_{L^p(X,\nu)}
        \le C\|f\|_{L^p(X,\nu)}.
\]
Finally, letting $R_0\to\infty$ and using monotone convergence gives

Theorem \ref{t:avg}. 

The same argument applies to the singular integral averages, since for
$R, R'\le R_0$ the difference of two truncated singular integral operators at radii $R, R'$, as it appears in the variation norm, only depends on the values of the
orbit function on the compact set $B(0,R_0)$.

\end{document}